\documentclass[12pt,a4paper,leqno]{article}

\usepackage[centertags]{amsmath}
\usepackage{amsthm}
\usepackage{amssymb,exscale}
\usepackage{paralist}
\usepackage{latexsym}
\usepackage{graphicx}
\usepackage{hyperref}

\swapnumbers
\theoremstyle{definition}
\newtheorem{theorem}[equation]{Theorem}
\newtheorem{lemma}[equation]{Lemma}
\newtheorem{corollary}[equation]{Corollary}
\newtheorem{definition}[equation]{Definition}
\newtheorem{proposition}[equation]{Proposition}
\newtheorem{remark}[equation]{Remark}

\renewcommand{\phi}{\varphi}

\newcommand{\D}{\mathrm{d}}
\newcommand{\E}{\mathrm{e}}

\renewcommand{\(}{\bigl(}
\renewcommand{\)}{\bigr)\vphantom{)}}

\newcommand{\impl}{\;\Longrightarrow\;}

\newcommand{\vol}{\operatorname{vol}}
\newcommand{\width}{\operatorname{width}}

\newcommand{\const}{\operatorname{const}}

\newcommand{\eps}{\varepsilon}

\newcommand{\ga}{\gamma}

\newcommand{\de}{\delta}
\newcommand{\De}{\Delta}
\newcommand{\al}{\alpha}
\newcommand{\be}{\beta}
\newcommand{\cO}{\mathcal O}

\newcommand{\la}{\lambda}

\newcommand{\Ex}{\mathbb E\,}
\newcommand{\R}{\mathbb R}

\begin{document}

\title{Linear response and moderate deviations:\\ hierarchical
  approach. II}

\author{Boris Tsirelson}

\date{}
\maketitle

\begin{abstract}
The Moderate Deviations Principle (MDP) is well-understood for sums of
independent random variables, worse understood for stationary random
sequences, and scantily understood for random fields.
An upper bound for a new class of random fields is obtained here by
induction in dimension.
\end{abstract}

\setcounter{tocdepth}{2}
\tableofcontents

\numberwithin{equation}{section}

\section[Definition, and main result formulated]
  {\raggedright Definition, and main result formulated}
\label{sect1}
We examine a class of random fields $ X = (X_t)_{t\in\R^d}
$, but we are interested only in integrals $ \int_B X_t \, \D t $ over
boxes $ B = [\al_1,\be_1]\times\dots\times[\al_d,\be_d] \subset \R^d
$ (where $ \al_1<\be_1,\dots,\al_d<\be_d $) rather than ``individual''
random variables $ X_t $. Similarly to \cite[Sect.~1]{I} we merely
deal with a box-indexed family of random variables, denoted (if only
for convenience) by $ \int_B X_t \, \D t $ and satisfying additivity:
\begin{multline}\label{01.1}
\int_{[\al,\be]\times B} X_t \, \D t + \int_{[\be,\ga]\times B}
 X_t \, \D t = \\
= \int_{[\al,\ga]\times B} X_t \, \D t
 \quad \text{a.s.\ whenever } -\infty<\al<\be<\ga<\infty \, ,
\end{multline}
for all boxes $ B \subset \R^{d-1} $; this being additivity w.r.t.\
the first coordinate, the same is required for each coordinate.
We say that $X$ is \emph{centered,} if
\begin{equation}\label{01.3}
\Ex \bigg| \int_B X_t \, \D t \bigg| < \infty \quad \text{and} \quad \Ex
\int_B X_t \, \D t = 0 \quad \text{for all boxes } B \subset \R^d \, .
\end{equation}
Two more conditions, stationarity and mesurability, will be used
in \cite{III} when proving the moderate deviations principle.
Stationarity means measure preserving time shifts that send $ \int_B
X_t \, \D t $ to $ \int_{B+s} X_t \, \D t $. Measurability (under
stationarity) is similar to \cite[(1.2)]{I}:
\begin{multline}\label{01.2}
\text{the distribution of} \int_{[0,r_1]\times\dots\times[0,r_d]}
 X_t \, \D t \\
\text{ is a Borel measurable function of } (r_1,\dots,r_d) \, . 
\end{multline}
But for now, our goal being upper bounds, we need only \eqref{01.1}
and \eqref{01.3}, and call such $X$ a \emph{centered random field} (on
$\R^d$). 
The notions ``independent'' and ``identically distributed'' are
interpreted for such processes similarly to \cite[Sect.~1]{I}.
Sometimes, when convenient, we write $X(t)$ rather than $X_t$.

Splittability, defined in \cite[Def.~1.4]{I} for $ d=1 $, will be
defined here for all $d$.

First, given a centered random field $ X = (X_t)_{t\in\R^d} $, we define
a \emph{split} of $X$ as a triple of random fields $ X^0, X^-, X^+ $
(on some probability space) such that the two fields $ X^-, X^+ $ are
(mutually) independent and the four fields $ X, X^0, X^-, X^+ $ are
identically distributed. (Informally, a split is useful when its leak
defined below is small.) Clearly, $ X^0, X^-, X^+ $ are centered (since $X$
is).

Second, given $ r,a,b \in \R $, $ a<r<b $, and a split $ (X^0, X^-,
X^+) $ of $X$, we define the \emph{leak} (of this split) along the
hyperplane $ \{r\}\times\R^{d-1} $ on the strip $ [a,b] \times \R^{d-1}
$ as the random field $ (Y_t)_{t\in\R^{d-1}} $ where
\begin{multline*}
Y(t_2,\dots,t_d) = \\
\int_{a}^r X^- (t_1,t_2,\dots,t_d) \, \D t_1
+ \int_r^b X^+ (t_1,t_2,\dots,t_d) \, \D t_1 \,
- \int_a^b X^0 (t_1,t_2,\dots,t_d) \, \D t_1
\end{multline*}
in the sense that
\[
\int_B Y(t) \, \D t = \int_{[a,r]\times B} X^-(t) \, \D t
+ \int_{[r,b]\times B} X^+(t) \, \D t - \int_{[a,b]\times B}
X^0(t) \, \D t
\]
for all boxes $ B \subset \R^{d-1} $. Clearly, $Y$ is also a centered
random field. Similarly, for each $ k \in \{1,\dots,d\} $ we define
the leak along the hyperplane $ \R^{k-1}\times\{r\}\times\R^{d-k} $ on
the coordinate strip $ \R^{k-1} \times [a,b] \times \R^{d-k}
$. (Different $k$ and $r$ need different splits, to be useful.)

Given $X$, we need splits $ (X_{k,r}^0, X_{k,r}^-, X_{k,r}^+) $ of $X$
for all $ k=1,\dots,d $ and $ r \in \R $; and for all $ a \in
(-\infty,r) $, $ b \in (r,\infty) $ we consider the corresponding leak
$ Y_{k,r,a,b} $ of this split.
For $d=1$ the leak is just a single random variable $ Y = \int_a^r
X^-(t) \, \D t + \int_r^b X^+(t) \, \D t - \int_a^b X^0(t) \, \D t $.

Accordingly, a family $ (X_i)_{i\in I} $ of random fields on $ \R^d
$ leads to a family $ (Y_{k,r,a,b,i})_{k=1,\dots,d;a<r<b;i\in I} $ of
centered random fields on $ \R^{d-1} $.

Third, we consider a family of centered random fields $ (X_i)_{i\in I} = \(
(X_{t,i})_{t\in\R^d} \)_{i\in I} $ (the index set $I$ being arbitrary)
and define uniform splittability of such family. Then splittability of
a centered random field appears as the particular case of a single-element
set $I$. We define uniform splittability by recursion in the dimension
$ d = 0,1,2,\dots $, treating a centered random field on $ \R^0 $ as just a
single random variable of zero mean.

\begin{definition}\label{definition1}
A family $ (X_i)_{i\in I} $ of centered random fields $ X_i =
(X_{t,i})_{t\in\R^d} $ is \emph{uniformly splittable,} if either $ d=0
$ and
\[
\exists \eps>0 \; \forall i\in I \;\> \Ex \exp \eps |X_i| \le 2 \, ,
\]
or $ d \ge 1 $ and the following two conditions hold:

(a) there exist a box $ B \subset \R^d $ and $ \eps>0 $ such that
\[
\forall i\in I \;\> \forall
s \in \R^d \;\; \Ex \exp \eps \bigg| \int_{B+s} X_i(t) \, \D
t \bigg| \le 2 \, ,
\]

(b) there exist splits $ (X_{k,r,i}^0, X_{k,r,i}^-, X_{k,r,i}^+) $ of
each $X_i$, whose leaks are a uniformly
splittable family $ (Y_{k,r,a,b,i})_{k=1,\dots,d;a<r<b;i\in I}$.
\end{definition}

We use volume and width of a box $ B =
[\al_1,\be_1]\times\dots\times[\al_d,\be_d] \subset \R^d $:
\[
\vol B = (\be_1-\al_1) \dots (\be_d-\al_d) \, , \quad
\width B = \min (\be_1-\al_1, \dots, \be_d-\al_d) \, .
\]

The theorem below applies first of all to a single random field (that
is, a single-element set $I$); the general formulation enables the
proof by induction in the dimension.

\begin{theorem}\label{th}
For every uniformly splittable family $ (X_i)_{i\in I} $ of centered random
fields $ X_i = (X_{t,i})_{t\in\R^d} $ there exists $ C \in (1,\infty)
$ such that for every $ i \in I $, every box $ B \subset \R^d $ and
every $ \la \in \R $,
\begin{gather*}
\text{if} \;\>\, C |\la| \le \frac1{ \log^d \vol B } \;\; \text{ and
} \; \width B \ge C \, , \quad \text{then} \\
\log \Ex \exp \la \int_B X_{t,i} \, \D t \le C (\vol B) \la^2 \, .
\end{gather*}
\end{theorem}

(Of course, $ \log^d \vol B $ is $ \( \log (\vol B) \)^d $.)
This theorem will be proved by induction in the dimension $ d =
1,2,\dots $ 
Throughout we assume that a uniformly splittable family $ (X_i)_{i\in
I} $ of centered random fields $ X_i = (X_{t,i})_{t\in\R^d} $ is given.
We assume that the theorem holds in dimension $d-1$, unless $d=1$; in
the latter case the proof needs trivial modifications.

According to Def.~\ref{definition1} we have splits (of the fields $
X_i $) whose leaks $ Y_{k,r,a,b,i} $ are a uniformly splittable family
(in dimension $d-1$). Theorem \ref{th}, applied to this family, gives
$C_1$ such that for all $ k,r,a,b,i $
\begin{equation}\label{1.6}
\log \Ex \exp \la \int_B Y_{k,r,a,b,i}(t) \, \D t \le C_1 (\vol B) \la^2
\end{equation}
whenever $ B \subset \R^{d-1} $ is a box of width $ \ge C_1 $, and $
C_1 |\la| \le \frac1{ \log^{d-1} \vol B } $.

For $ d=1 $ the leaks $ Y_{1,r,a,b,i} $ are just random
variables; their uniform splittability means
$ \exists \eps>0 \; \forall r,a,b \; \forall i \in
I \;\> \Ex \exp \eps|Y_{1,r,a,b,i}| \le 2 $. By Lemma \ref{1a8} below this
implies $ \log \Ex \exp \eps\la Y_{1,r,a,b,i} \le \la^2 $ for $ \la \in
[-1,1] $. Taking $ C_1 = \max \( \frac1{\eps^2}, 1 \) $ we get for all
$ r,a,b,i $
\begin{equation}\label{1.7}
\log \Ex \exp \la Y_{1,r,a,b,i} \le C_1 \la^2 \quad \text{whenever } C_1
|\la| \le 1 \, ,
\end{equation}
to be used for $ d=1 $ instead of \eqref{1.6}.

We borrow from \cite[Lemma 2a8]{I} a general fact.

\begin{lemma}\label{1a8}
If a random variable $Z$ satisfies $ \Ex \exp |Z| \le 2 $ and $ \Ex Z
= 0 $, then $ \log \Ex \exp \la Z \le \la^2 $ for all $ \la \in [-1,1]
$.
\end{lemma}

Theorem \ref{th} is proved in Sect.~\ref{sect2} for $ \la = \cO \(
(\vol B)^{-\frac1{2d}} \log^{-(d-1)} \vol B \) $; larger $\la$ are
treated in Sect.~\ref{sect3}. I still do not know, what happens when
$\la$ tends to $0$ slower than $ \log^{-d} \vol B $. This logarithmic
gap between moderate and large deviations, is it a phenomenon or a
drawback of my approach?

\section[Far from large deviations]
  {\raggedright Far from large deviations}
\label{sect2}
\begin{proposition}\label{2a1}
There exists $ C \in (1,\infty)
$ such that for every $ i \in I $, every box $ B \subset \R^d $ of
volume $ v $ and width $ \ge C $, and every $ \la \in \R $,
\[
\text{if} \;\>\, C |\la| \le \frac1{ v^{1/(2d)} \log^{d-1} v } \,
, \quad \text{then} \;\>
\log \Ex \exp \la \int_B X_i(t) \, \D t \le C v \la^2 \, .
\]
\end{proposition}

Similarly to \cite[Sect.~2a]{I} we consider random variables
\[
S_{B,i} = \frac1{\sqrt{\vol B}} \int_B X_{t,i} \, \D t
\]
their cumulant generating functions $ \la \mapsto \log \Ex \exp \la
S_{B,i} $, and ensure shift invariance by taking the supremum over
all shifts of a box:
\begin{equation}\label{2a3}
f_{B,i}(\la) = \sup_{s\in\R^d} \log \Ex \exp \la S_{B+s,i} \, .
\end{equation}
Still, $ f_{B,i}(\la) \ge 0 $, since $ \Ex \exp \la S_{B,i} \ge \Ex(1+\la
S_{B,i}) = 1 $.

Further, we take the supremum over all $i$ and all boxes of a given
volume and width $ \ge C $:
\begin{gather*}
f_B(\la) = \sup_{i\in I} f_{B,i} (\la) \, ; \\
f_{v,C}(\la) = \sup_{\vol(B)=v,\width B\ge C} f_B (\la) \quad
\text{for } v \ge C^d \, .
\end{gather*}
All these functions map $ \R $ to $ [0,\infty] $.

Denoting for convenience
\[
R(v) = v^{\frac1d} \quad \text{and} \quad S(v) = v^{\frac{d-1}d}
\]
we rewrite (not proved yet) Prop.~\ref{2a1} as follows.

\begin{proposition}\label{2a4}
There exists $ C \in (1,\infty) $ such that
\[
f_{v,C}(\la) \le C \la^2 \quad \text{whenever } \; C |\la| \le \frac{
  \sqrt{S(v)} }{ \log^{d-1} v } \text{ and } v \ge C^d \, .
\]
\end{proposition}

We generalize \cite[Prop.~2a9(a)]{I}. Given a box $ B \subset \R^{d-1} $
and a number $ r>0 $, we consider two boxes $ B_1 = [0,r] \times B $
and $ B_2 = [-r,r] \times B $ in $ \R^d $. Let $ v = \vol B_1 $ and $
\width B \ge C_1 $. For $d=1$ we mean $ B_1 = [0,r] $ and $ B_2 =
[-r,r] $; $B$ disappears, as well as the condition on $ \width B $; by
convention, $ \log^{-(d-1)} \vol B = 1 $, and (in the proof) $ \vol B
= 1 $.

\begin{lemma}\label{2a5}
For all $ p \in (1,\infty) $ and $ \la $ such that
$ C_1 |\la| \le \frac{p-1}p \sqrt{2v} \log^{-(d-1)} \vol B $,
\[
f_{B_2} (\la) \le \frac2p f_{B_1} \Big( \frac{p\la}{\sqrt2} \Big) +
C_1 \frac{p}{p-1} \cdot \frac{\la^2}{2r} \, .
\]
\end{lemma}

\begin{proof}
Given $ i \in I $, we use the split $ (X^0, X^-, X^+) $ of $X_i$ whose
leak $ Y = Y_{1,0,-r,r,i} $ on the strip $ [-r,r] \times \R^{d-1} $
satisfies \eqref{1.6}. Similarly to \cite[2a7 and 2a9]{I}, the random
variables $ U = \frac1{\sqrt v} \int_{[-r,0]\times B} X_t^- \, \D t $,
$ V = \frac1{\sqrt v} \int_{[0,r]\times B} X_t^+ \, \D t $,
$ W = \frac1{\sqrt{2v}} \int_{[-r,r]\times B} X_t^0 \, \D t $ and $ Z
= - \int_B Y_t \, \D t $ satisfy $ Z = \sqrt{2v} W - \sqrt v U - \sqrt
v V $ and, by \eqref{1.6}, $ \log \Ex \exp \la Z \le C_1 (\vol B) \la^2 $ for $ C_1
|\la| \le \log^{-(d-1)} \vol B $. Similarly to \cite[Prop.~2a9]{I}, by
H\"older's inequality, $ f_{B_2,i} (\la) = \log \Ex \exp
\frac{\la}{\sqrt{2v}} \int_{B_2} X_{t,i} \, \D t = \log \Ex \exp \la W
= \log \Ex \exp \la \( \tfrac{U+V}{\sqrt2} + \frac1{\sqrt{2v}} Z \)
\le \frac1p \cdot 2f_{B_1} \( \frac{p\la}{\sqrt2} \) + \frac{p-1}p
\log \Ex \exp \frac{p}{p-1} \frac{\la Z}{\sqrt{2v}} $. The second
term does not exceed $ \frac{p-1}p C_1 (\vol B) \( \frac{p}{p-1}
\frac{\la}{\sqrt{2v}} \)^2 = C_1 \frac{\vol B}{2v} \frac{p}{p-1}
\la^2 $ for $ \frac{p}{p-1} \frac{C_1|\la|}{\sqrt{2v}} \le \log^{-(d-1)}
\vol B $; and $ \frac{\vol B}{2v} = \frac1{2r} $. It remains to take
supremum in $ i \in I $.

For $d=1$ the leak $ Y = Y_{1,0,-r,r,i} $, being a random variable,
satisfies \eqref{1.7}; $ Z = -Y $; $ \log \Ex \exp \la Z \le C_1 \la^2
$ for $ C_1 |\la| \le 1 $ (and $ v=r $, of course).
\end{proof}

\begin{remark}\label{2a5a}
Above, a box is halved (divided in two boxes) by the hyperplane $
x_1=0 $. More generally, the same
holds when halving the box by another coordinate hyperplane $ x_k = c $
for $ k\in\{1,\dots,d\} $ and the appropriate $ c $.
\end{remark}

Now we consider two boxes in $ \R^d $, $ B_0 = [0,r_1] \times \dots
\times [0,r_d] $ and $ B = [0,2^{n_1}r_1] \times \dots \times
[0,2^{n_d}r_d] $ for arbitrary $ r_1,\dots,r_d \in [C,2C) $ for some $
C \ge C_1 $, and arbitrary $ n_1,\dots,n_d \in \{0,1,2,\dots\} $.

\begin{proposition}\label{2a6}
If $C$ is large enough, then for all $ r_1,\dots,r_d \in [C,2C) $, $
\de>0 $ and $ a \ge \frac{C}{R(\vol B_0)} $ satisfying
\[
f_{B_0} (\la) \le a \la^2 \quad \text{whenever } |\la| \le \de \, ,
\]
where $ B_0 = [0,r_1] \times \dots \times [0,r_d] $, there exists
natural $N$ such that the following holds for all $ n_1,\dots,n_d \in
\{0,1,2,\dots \} $ satisfying $ n_1+\dots+n_d \ge N $:
\[
f_B (\la) \le 2a \la^2 \quad \text{whenever } |\la| \le \De \, ,
\]
where $ B = [0,2^{n_1} r_1] \times \dots \times [0,2^{n_d} r_d] $ and
$ \De = \frac1{C_1} \sqrt{ \frac1a S(\vol B) } \log^{-(d-1)} S(\vol B)
$.
\end{proposition}

(For $d=1$, by convention, $ \log^{-(d-1)} S(\dots) = 1 $,
notwithstanding that $ S(\dots)=1 $.)

\begin{proof}
We take $ n = n_1+\dots+n_d $, $ B_n = B $, halve the longest side of
$ B_n $, denote the half of $ B_n $ by $ B_{n-1} $, and repeat this
operation getting $ B_{n-2}, \dots, B_0 $. For each $ k = 0,\dots,n-1
$ we have $ \vol B_{k+1} = 2^{k+1} \vol B_0 $ and, assuming $ C \ge
C_1 $,
\begin{gather*}
\!\!\!\!\!\! f_{B_{k+1}} (\la) \le \frac2p f_{B_k} \Big( \frac{p\la}{\sqrt2} \Big) +
C_1 \frac{p}{p-1} \cdot \frac{\la^2}{R(\vol B_{k+1})}  \\
\qquad \text{for } C_1 |\la| \le \frac{p-1}p \sqrt{\vol B_{k+1}}
\log^{-(d-1)} S(\vol B_{k+1})
\end{gather*}
by Lemma~\ref{2a5} (and Remark \ref{2a5a}), since (recall $r$ and $B$ of \ref{2a5}) the
longest side $ 2r $ of $ B_{k+1} $ cannot be less than $ R(\vol
B_{k+1}) $, and $ \vol B = \frac{ \vol B_{k+1} }{ 2r } \le S(\vol
B_{k+1}) $. (For $d=1$ this is just $ 1 = \frac{2r}{2r} \le 1 $.)

Given $ A_k $ and $ \De_k $ such that $ f_{B_k}(\la) \le A_k^2 \la^2 $
for $ |\la| \le \De_k $, we denote $ q=\frac{p}{p-1} $, $ x = \sqrt{
\frac{C_1}{R(\vol B_{k+1})} } $ and get
\[
f_{B_{k+1}} (\la) \le \frac2p A_k^2 \Big( \frac{p\la}{\sqrt2} \Big)^2
+ q x^2 \la^2 = ( A_k^2 p + x^2 q ) \la^2
\]
for $ |\la| \le \min \( \frac{\sqrt2}{p} \De_k, \frac1{C_1 q}
\sqrt{\vol B_{k+1}} \log^{-(d-1)} S(\vol B_{k+1}) \) $. Generally, the
minimum of $ A_k^2 p + x^2 q $ over $p,q$ such that $ \frac1p +
\frac1q = 1 $ is equal to $ (A_k+x)^2 $, and is reached at $ p = 1 +
\frac{x}{A_k} $, $ q = 1 + \frac{A_k}x $. Thus, $ f_{B_{k+1}}(\la) \le
A_{k+1}^2 \la^2 $ for $ |\la| \le \De_{k+1} $, provided that $ A_{k+1}
\ge A_k + x = A_k + \sqrt{ \frac{C_1}{R(\vol B_{k+1})} } $ and
$ \De_{k+1} \le \min \Big( \frac{\sqrt2}{p_k} \De_k, \frac1{C_1 q_k}
\sqrt{\vol B_{k+1}} \log^{-(d-1)} S(\vol B_{k+1}) \Big) $; here $ p_k
= 1 + \frac{x}{A_k} = 1 + \frac{1}{A_k} \sqrt{ \frac{C_1}
{R(\vol B_{k+1})} } $ and $ q_k = 1 + \frac{A_k}{x} = 1 +
\frac{A_k}{\sqrt{C_1}} \sqrt{ R(\vol B_{k+1})} $.

We take $ A_k = \sqrt a + \sqrt{ \frac{C_1}{R(\vol B_0)} }
\sum_{i=1}^k 2^{-\frac{i}{2d}} $ (thus $ A_{k+1} = A_k + \sqrt{
\frac{C_1}{R(\vol B_{k+1})} } \, $) and note that $ A_k \uparrow A_\infty
\le \sqrt a + \sqrt{ \frac{C_1}{C} a } \sum_{i=1}^\infty
2^{-\frac{i}{2d}} \le \sqrt{2a} $ if $C$ is large enough (since $ a
\ge \frac{C}{R(\vol B_0)} $). Also, $ q_k \le 1 +
\frac{\sqrt{2a}}{\sqrt{C_1}} \sqrt{R(\vol B_{k+1})} \le
\frac{\sqrt{3a}}{\sqrt{C_1}} \sqrt{R(\vol B_{k+1})} $ for all $k$, if
$C$ is large enough (since $ a \ge \frac{C}{R(\vol B_0)} $
again). Assuming also $ C_1 \ge 3 $ (which is harmless) we introduce $
M_k = \frac1{C_1} \sqrt{\frac1a S(\vol B_k)} \log^{-(d-1)} S(\vol B_k)
$, note that $ M_{k+1} \le \frac1{C_1 q_k} \sqrt{\vol B_{k+1}}
\log^{-(d-1)} S(\vol B_{k+1}) $ (since $ q_k \le \sqrt{ a R(\vol
B_{k+1}) } $), and replace the condition on $ \De_{k+1} $ given
above with the stronger condition $ \De_{k+1} \le \min \(
\frac{\sqrt2}{p_k} \De_k, M_{k+1} \) $. Now we note that $ p_k - 1 =
\frac1{A_k} \sqrt{ \frac{C_1}{R(\vol
B_{k+1})} } \le \frac1{\sqrt a} \sqrt{ \frac{C_1}{R(\vol
B_{k+1})} } \le \sqrt{ \frac{C_1 R(\vol B_0)}{C R(\vol B_{k+1})} }
\le 2^{-\frac{k+1}{2d}} $ if $ C \ge C_1 $, take integer $N$ such that
$ 2^{-\frac{N+1}{2d}} \le 2^{\frac1{2d}} - 1 $, and get $ M_{k+1} \le
\frac{\sqrt2}{p_k} M_k $ for all $ k \ge N $ (since $
\frac{M_{k+1}}{M_k} \le 2^{\frac{d-1}{2d}} $ and $ p_k \le
1+2^{-\frac{N+1}{2d}} \le 2^{\frac1{2d}} $).
We choose $ \De_k $ as follows:
\begin{align*}
\De_k& = M_k &\text{for } k \ge N \, , \\
\De_k &= \frac{p_k}{\sqrt2} \frac{p_{k+1}}{\sqrt2} \dots
\frac{p_{N-1}}{\sqrt2} M_N &\text{for } k < N \, .
\end{align*}
Clearly, $ \De_{k+1} \le \frac{\sqrt2}{p_k} \De_k $ for all $k$.

In order to obtain $ f_{B_n}(\la) \le 2a\la^2 $ for $ |\la| \le M_n $
when $ n\ge N $ it is sufficient to ensure that $ \De_0 \le \de $ and
$ \De_k \le M_k $ for $ k=0,1,\dots,N-1 $. We note that $ p_0 \dots
p_{N-1} \le \prod_{k=0}^{N-1} (1+2^{-\frac{k+1}{2d}}) \le \exp
\sum_{k=0}^\infty 2^{-\frac{k+1}{2d}} $ and $ C_1 M_N \le \sqrt{ \frac1a
S(\vol B_N) } \le \sqrt{ \frac1C R(\vol B_0) S(\vol B_N) } =
2^{\frac{d-1}{2d}N} \sqrt{ \frac1C \vol B_0 } $, thus $ \De_0 =
p_0 \dots p_{N-1} \cdot 2^{-\frac N 2} M_N \le \( \exp
\sum_k 2^{-\frac{k+1}{2d}} \) \cdot 2^{-\frac{N}{2d}} \frac1{C_1}
\sqrt{ \frac1C (2C)^d } $; by increasing $N$ as needed we get $ \De_0
\le \de $.

It remains to ensure that $ \De_k \le M_k $ for $ k=0,1,\dots,N-1 $.
We'll get a bit more: $ \De_k \le \frac1{C_1} \sqrt{ \frac1a
S(\vol B_k) } \log^{-(d-1)} S(\vol B_N) $, that is,
\begin{gather*}
\frac{p_k}{\sqrt2} \frac{p_{k+1}}{\sqrt2} \dots \frac{p_{N-1}}{\sqrt2}
 \sqrt{ \frac1a S(\vol B_N) } \le \sqrt{ \frac1a S(\vol B_k) } \, ; \\
p_k \dots p_{N-1} \le 2^{\frac{N-k}2} \cdot 2^{-\frac{d-1}{2d}(N-k)} =
2^{\frac{N-k}{2d}} \, .
\end{gather*}
We may check it only for $ k=0 $ and $ k=N $ due to the fact that $
p_k \dots p_{N-1} $ is a logarithmically convex function of $k$ (since
$p_k$ decrease). For $k=N$ it is just $1\le1$. For $k=0$ we need $ p_0
\dots p_{N-1} \le 2^{\frac{N}{2d}} $, which holds for $N$ such that $ 
\exp \sum_k 2^{-\frac{k+1}{2d}} \le 2^{\frac{N}{2d}} $.
\end{proof}

\begin{remark}
In the proof of \ref{2a6}, the restriction on $C$ depends only on $d$
and $C_1$. For large $d$, roughly, $ C \ge \cO(d^2) C_1 $. Also, the
restriction on $N$ depends only on $d$, $C$ and $\de$; roughly, $ N
\ge \cO(d^2) \log C + \cO(d) \log \frac1\de $.
\end{remark}

\begin{remark}\label{2a8}
In \ref{2a6}, $a$ and $\de$ may depend on $ r_1,\dots,r_d $. Assume
for a while that they do not; that is, the given $a$ and $\de$ serve
all $ r_1,\dots,r_d $ (for the given $C$). Then the conclusion (that 
$ f_B (\la) \le 2a \la^2 $ whenever $ |\la| \le \De $) holds for all
$B$ such that $ \width B \ge C $ and $ \vol B $ is large enough
(namely, $ \vol B \ge 2^N (2C)^d $). We get, for all $v$ large enough,
\[
f_{v,C} (\la) \le 2a \la^2 \quad \text{whenever } C_1 |\la| \le \sqrt{
  \frac1a S(v) } \log^{-(d-1)} S(v) \, .
\]
\end{remark}

It appears (Lemma~\ref{2a12} below) that the assumption of \ref{2a8}
is satisfied always (that is, for every uniformly splittable
family). Alternatively, the reader may just include that assumption
into Def.~\ref{definition1} (replacing Item (a) there) and skip to the
proof of Prop.~\ref{2a4} near the end of this section.

\begin{lemma}\label{2a9}
For every uniformly splittable family $ (X_i)_{i\in I} $ of centered random
fields on $ \R^d $ and every box $ B \subset \R^d $ there exist $ a,
\de > 0 $ such that
\[
f_B(\la) \le a \la^2 \quad \text{whenever } |\la| \le \de \, .
\]
\end{lemma}

This lemma will be proved by induction in the dimension $ d =
1,2,\dots $ 
As was noted near \eqref{1.6}, the given family $ X =
(X_i)_i $ on $ \R^d $ leads to another family $ Y = (Y_{k,r,a,b,i}) $ on
$ \R^{d-1} $. Both families of random fields lead to box-indexed
families of functions $ \R \to [0,\infty] $; $X$ leads to $
(f_B)_{B\subset\R^d} $ as before; likewise, $Y$ leads to $
(g_B)_{B\subset\R^{d-1}} $. (For $d=1$, just a single function $g$.)

Similarly to \cite[2a9(b) and 2a10]{I} we modify Lemma~\ref{2a5} as
follows. Given a box $ B_0 \subset \R^{d-1} $ and numbers $ r,s>0 $,
we consider three boxes $ B_1 = [-r,0] \times B_0 $, $ B_2 = [0,s]
\times B_0 $ and $ B = [-r,s] \times B_0 $ in $ \R^d $.

\begin{lemma}\label{2a10}
For all $ p \in (1,\infty) $ and $ \la \in \R $,
\begin{align*}
f_B (\la) &\le
\frac1p f_{B_1} \bigg( p\la \sqrt{\frac{r}{r+s}} \, \bigg) +
\frac1p f_{B_2} \bigg( p\la \sqrt{\frac{s}{r+s}} \, \bigg) +
\frac{p-1}{p} g_{B_0} \bigg( \frac{p}{p-1} \frac{-\la}{\sqrt{r+s}}
\bigg) \, ; \\
f_B (\la) &\ge
p f_{B_1} \bigg( \frac{\la}{p} \sqrt{\frac{r}{r+s}} \, \bigg) +
p f_{B_2} \bigg( \frac{\la}{p} \sqrt{\frac{s}{r+s}} \, \bigg) -
(p-1) g_{B_0} \bigg( \frac{1}{p-1} \frac{\la}{\sqrt{r+s}}
\bigg) \, .
\end{align*}
\end{lemma}

\begin{proof}
Similar to the proof of \ref{2a5}. Denoting $ v = \vol B_0 $,
the random variables
$ U = \frac1{\sqrt{rv}} \int_{B_1} X_t^- \, \D t $,
$ V = \frac1{\sqrt{sv}} \int_{B_2} X_t^+ \, \D t $,
$ W = \frac1{\sqrt{(r+s)v}} \int_B X_t^0 \, \D t $ and $ Z
= - \int_{B_0} Y_t \, \D t $ satisfy $ \sqrt{r+s} W = \sqrt r U +
\sqrt s V + \frac1{\sqrt v} Z $. By H\"older's inequality, $ f_{B,i}
(\la) = \log \Ex \exp \la W = \log \Ex \exp \la \(
\sqrt{\frac{r}{r+s}} U + \sqrt{\frac{s}{r+s}} V +
\frac1{\sqrt{(r+s)v}} Z \) \le \frac1p \log \Ex \exp p \la \(
\sqrt{\frac{r}{r+s}} U + \sqrt{\frac{s}{r+s}} V \) + \frac{p-1}p \log
\Ex \exp \frac{p}{p-1} \la \frac1{\sqrt{r+s}} \frac1{\sqrt v} Z \le
\linebreak
\frac1p f_{B_1} \( p\la \sqrt{\frac{r}{r+s}} \, \) + \frac1p f_{B_2}
\( p\la \sqrt{\frac{s}{r+s}} \, \) + \frac{p-1}{p} g_{B_0} \( -
\frac{p}{p-1} \frac{\la}{\sqrt{r+s}} \) $; supremum in $ i $ gives the
first inequality (the upper bound). The second inequality (the lower
bound), being rewritten as $ f_{B_1} \( \la \sqrt{\frac{r}{r+s}} \, \)
+ f_{B_2} \( \la \sqrt{\frac{s}{r+s}} \, \) \le \frac1p f_B(p\la) +\linebreak
\frac{p-1}{p} g_{B_0} \( \frac{p}{p-1} \frac{\la}{\sqrt{r+s}} \) $,
follows by H\"older's inequality from the relation $ \sqrt r U + \sqrt
s V = \sqrt{r+s} W - \frac1{\sqrt v} Z $.
\end{proof}

\begin{remark}\label{2.11}
Above, a box $B$ is divided in two boxes $B_1,B_2$ by the hyperplane $
x_1=0 $. More generally, the same holds when $B$ is divided by another
coordinate hyperplane $ x_k = c $ for $ k\in\{1,\dots,d\} $ and any
appropriate $ c $.
\end{remark}

Let us call a box $ B \subset \R^d $ \emph{good} when there exist $ a,
\de > 0 $ such that $ f_B(\la) \le a \la^2 $ whenever $ |\la| \le \de
$. Similarly, a box $ B \subset \R^{d-1} $ is good when $ \exists
a,\de>0 \; \forall\la \; \( |\la|\le\de \impl g_B(\la) \le a \la^2 \)
$.

Existence of (at least one) good box follows from Item (a) of
Def.~\ref{definition1} and Lemma \ref{1a8}. (For $d=0$ the only
``box'' is good.)

In order to prove Lemma~\ref{2a9} we assume (the induction hypothesis)
that all boxes in $ \R^{d-1} $ are good, and prove that all boxes in $
\R^d $ are good.

By \ref{2a10}, $B$ is good if and only if $ B_1, B_2 $ are good. It
follows that every box contained in some good box is good (turn from
$B$ to $B_1$ or $B_2$, and iterate).

Thus, all boxes that are small enough are good. It follows that every
box is good (divide it into small boxes).

Given a set of boxes, we say that these boxes are \emph{uniformly
good} when there exist $ a,\de > 0 $ such that for every box $B$ of
the given set, every $ \la \in [-\de,\de] $ satisfies the inequality $
f_B(\la) \le a \la^2 $ (or $ g_B(\la) \le a \la^2 $, for $ B \subset
\R^{d-1} $).

\begin{lemma}\label{2a12}
Let $ 0 < c < C < \infty $. Then the boxes $ [0,r_1] \times \dots
\times [0,r_d] $ for all $ r_1,\dots,r_d \in [c,C] $ are uniformly
good.
\end{lemma}

\begin{proof}
Induction in the dimension $d$. The induction hypothesis gives $ a_0,
\de_0 > 0 $ such that $ g_{[0,r_1] \times \dots \times [0,r_{d-1}] }
(\la) \le a_0 \la^2 $ whenever $ |\la| \le \de_0 $ and $
r_1,\dots,r_{d-1} \in [c,C] $.
(When $d=1$, this holds for the single function $g$.)
The box $ [0,C]^d = [0,C] \times \dots
\times [0,C] $ being good by \ref{2a9}, we take $ a,\de > 0 $ such
that $ f_{[0,C]^d} (\la) \le a \la^2 $ whenever $ |\la| \le \de $. We
use the second inequality of \ref{2a10} for $ p=2 $ (taking \ref{2.11}
into account):
\begin{gather*}
2f_{[0,r_1]\times[0,C]^{d-1}} \Big( \frac\la2 \sqrt{\frac{r_1}C} \Big)
  \le f_{[0,C]^d} (\la) + g_{[0,C]^{d-1}} \Big( \frac\la{\sqrt C}
  \Big) \, ; \\
\!\!\!\!\!\!\!\!\!\!\!\!
  f_{[0,r_1]\times[0,C]^{d-1}} (\la) \le \frac12 f_{[0,C]^d} \Big(
  2\la \sqrt{\frac{C}{r_1}} \Big) + \frac12 g_{[0,C]^{d-1}} \Big(
  \frac{2\la}{\sqrt{r_1}} \Big) \le \\
\qquad\qquad\qquad \le \frac12 a \cdot 4\la^2 \frac{C}{r_1} + \frac12 a_0 \cdot
\frac{4\la^2}{r_1} \le \Big( \frac{2aC}c + \frac{2a_0}c \Big) \la^2
\end{gather*}
for $ |\la| \le \min \( \frac\de2 \sqrt{\frac{c}{C}}, \frac{\de_0}2
\sqrt c \) $. Thus, the boxes $ [0,r_1]\times[0,C]^{d-1} $ for $ r_1
\in [c,C] $ are uniformly good. Now we divide the box $
[0,r_1]\times[0,C]^{d-1} $ by the hyperplane $ x_2 = r_2 $, apply
again the argument used above, and see that the boxes $
[0,r_1]\times[0,r_2]\times[0,C]^{d-2} $ for $ r_1,r_2 \in [c,C] $ are
uniformly good. And so on.
\end{proof}

\begin{proof}[Proof of Prop.~\ref{2a4}]
We take $C$ large enough according to \ref{2a6}. By \ref{2a12} the
boxes $ B = [0,r_1] \times \dots \times [0,r_d] $ for all $
r_1,\dots,r_d \in [C,2C] $ are uniformly good. We take $ a \ge 1 $ and $\de>0 $
such that $ f_B(\la) \le a\la^2 $ for all these $B$ and all $ \la \in
[-\de,\de] $. Now \ref{2a8} gives $V$ such that all $ v \in [V,\infty)
$ satisfy
\[
f_{v,C} (\la) \le 2a \la^2 \quad \text{whenever } C_1 |\la| \le \sqrt{
  \frac1a S(v) } \log^{-(d-1)} S(v) \, .
\]
We take $ M>1 $ such that $ M \ge C $, $ M^d \ge V $, $ M \ge 2a $, $
M \ge C_1 \sqrt a $, and get
\[
f_{v,M}(\la) \le M \la^2 \quad \text{whenever } \; M |\la| \le \frac{
  \sqrt{S(v)} }{ \log^{d-1} v } \text{ and } v \ge M^d \, ,
\]
since $ v>1 $, $ \log^{d-1} S(v) \le \log^{d-1} v $ (just $ 1\le1 $
for $d=1$), $ v \ge M^d \ge V $, $ C_1 |\la| \le \frac{C_1}M \sqrt{S(v)}
\log^{-(d-1)} v \le \sqrt{ \frac1a S(v) } \log^{-(d-1)} S(v) $, and $
f_{v,M}(\la) \le f_{v,C}(\la) \le 2a \la^2 \le M \la^2 $.
\end{proof}

Proposition \ref{2a1} is thus proved.

\section[Close to large deviations]
  {\raggedright Close to large deviations}
\label{sect3}
In this section we denote by $C_2$ the constant given by
Prop.~\ref{2a4}, use the functions $ f_{v,C} $ mostly for $C=C_2$,
and denote $ f_v = f_{v,C_2} $. By \ref{2a4},
\begin{equation}\label{3.1}
f_v(\la) \le C_2 \la^2 \quad \text{whenever } \; C_2 |\la| \le \frac{
  \sqrt{S(v)} }{ \log^{d-1} v } \text{ and } v \ge C_2^d \, .
\end{equation}

We still use \eqref{1.6}, \eqref{1.7} and $C_1$ therefrom; $ C_1 \le
C_2 $. By convention, $ \log^0 x = 1 $ always (also if $x$ does not
belong to $(1,\infty) $ or even is ill-defined).

\begin{lemma}
For all $ p \in (1,\infty) $,
\[
f_{2v} (\la) \le \frac2p f_v \Big( \frac{p\la}{\sqrt2} \Big) +
C_1 \frac{p}{p-1} \cdot \frac{\la^2}{R(2v)}
\]
whenever $ C_1 |\la| \le \frac{p-1}p \sqrt{2v} \log^{-(d-1)} S(2v) $
and $ 2v \ge (2C_2)^d $. 
\end{lemma}

\begin{proof}
Given a box $ B $ such that $ \vol B = 2v $ and $ \width B \ge C_2 $,
we halve the longest side $2r$ of $B$ and apply Lemma~\ref{2a5} (as we
did in the proof of \ref{2a6}). Once again, $ 2r \ge R(\vol B) =
R(2v) \ge 2C_2 $, thus a half of $B$ is still of width $ \ge C_2 $,
and $ \frac{2v}{2r} \le S(2v) $, thus \ref{2a5} applies and gives $
f_B(\la) \le \frac2p f_v \( \frac{p\la}{\sqrt2} \) +
C_1 \frac{p}{p-1} \cdot \frac{\la^2}{2r} $.
\end{proof}

It follows that
\[
f_{2v} (\la) \le \frac2p f_v \Big( \frac{p\la}{\sqrt2} \Big) +
|\la| \frac{\sqrt{2v}}{R(2v)} \log^{-(d-1)} S(2v) \, ,
\]
since $ \frac{p}{p-1} C_1 |\la| \le \sqrt{2v} \log^{-(d-1)} S(2v) $.

For convenience we denote
\[
\phi_v(\la) = \frac1{ |\la| \sqrt v } f_v(\la) \quad \text{for } v \ge
C_2^d \text{ and } \la \ne 0 \, .
\]

\begin{corollary}\label{3.2}
For all $ p \in (1,\infty) $,
\[
\phi_{2v} (\la) \le \phi_v \Big( \frac{p\la}{\sqrt2} \Big) +
\frac1{ R(2v) \log^{d-1} S(2v) }
\]
whenever $ 0 < C_1 |\la| \le \frac{p-1}p \sqrt{2v} \log^{-(d-1)} S(2v)
$ and $ 2v \ge (2C_2)^d $. 
\end{corollary}

\begin{proof}
$ \phi_{2v}(\la) - \phi_v \( \frac{p\la}{\sqrt2} \) = \frac1{
|\la| \sqrt{2v} } f_{2v}(\la) - \frac{ \sqrt2 }{ p |\la| \sqrt v }
f_v \( \frac{p\la}{\sqrt2} \) \le \frac1{ |\la| \sqrt{2v} } \( \frac2p
f_v \( \frac{p\la}{\sqrt2} \) + |\la| \frac{\sqrt{2v}}{R(2v)}
\log^{-(d-1)} S(2v) \) - \frac{ \sqrt2 }{ p |\la| \sqrt v }
f_v \( \frac{p\la}{\sqrt2} \) = \frac1{ R(2v) \log^{d-1} S(2v) }
$.
\end{proof}

\begin{corollary}\label{3.3}
\[
\phi_{2v} (\la_1) \le \phi_v (\la_0) + \frac1{ R(2v) \log^{d-1}
S(2v) }
\]
whenever $ 2v \ge (2C_2)^d $, $ \la_0 \la_1 > 0 $ and
\[
\frac{ \sqrt2 }{ |\la_1| } - \frac1{ |\la_0| } \ge \frac{ C_1 }{ \sqrt
v } \log^{d-1} S(2v) \, .
\]
\end{corollary}

\begin{proof}
We take $ p = \frac{ \sqrt2 \la_0 }{ \la_1 } $ and note that $ \la_0
= \frac{ p \la_1 }{ \sqrt2 } $ and $ \frac{p-1}{p|\la_1|}
= \frac1{|\la_1|} - \frac1{\sqrt2 |\la_0|} \ge
\frac{C_1}{\sqrt{2v}} \log^{d-1} S(2v) $, that is, $ C_1
|\la_1| \le \frac{p-1}p \sqrt{2v} \log^{-(d-1)} S(2v) $,
thus \ref{3.2} applies.
\end{proof}

\pagebreak[2]

\begin{corollary}\label{3.5}
Let numbers $ v_0,\dots,v_n $ and $ \la_0,\dots,\la_n $ satisfy
\[
v_{k+1} = 2v_k \, , \quad \la_k \la_{k+1} > 0 \, , \quad 
\frac{ \sqrt2 }{ |\la_{k+1}| } - \frac1{ |\la_k| } \ge \frac{ C_1
}{ \sqrt v_k } \log^{d-1} S(2v_k)
\]
for $ k=0,\dots,n-1 $; and $ 2v_0 \ge (2C_2)^d $. Then
\[
\phi_{v_n} (\la_n) \le \phi_{v_0} (\la_0) +
\sum_{k=0}^{n-1} \frac1{ R(2v_k) \log^{d-1} S(2v_k) } \, .
\]
\end{corollary}

\begin{proof}
Just apply \ref{3.3} $n$ times.
\end{proof}

We rewrite Theorem \ref{th} in terms of $ f_v $.

\begin{proposition}\label{3.6}
There exists $ C \in (1,\infty) $ such that for every $ \la \in \R $,
\[
\text{if } C|\la| \le \frac{ \sqrt v }{\log^d v } \text{ and } v \ge
C^d, \quad \text{then} \quad f_{v,C} (\la) \le C \la^2 \, .
\]
\end{proposition}

Taking $ C \ge \max( C_2, \E^{1/d} ) $ we see that the case $ C
|\la| \le \frac{ \sqrt{S(v)} }{ \log^{d-1} v } $ is covered
by \eqref{3.1}.

From now on we assume that $ v \ge C^d $ and
\begin{equation}\label{3.**}
\frac{ \sqrt{S(v)} }{ \log^{d-1} v } < C|\la| \le \frac{ \sqrt v
}{\log^d v } \, ;
\end{equation}
ultimately we'll prove that $ f_{v,C} (\la) \le C \la^2 $ provided
that the constant $C$, dependent only on $ d, C_1, C_2 $, is large
enough.

We take integer $n$ such that
\begin{equation}\label{3.*}
2^{n-1} < M_d^{2d} \frac{ (C|\la|)^{2d} }{ v^{d-1} }
\log^{2d(d-1)} \frac{ \sqrt v }{ C|\la| } \le 2^n
\end{equation}
(the constant $M_d$, dependent on $d$ only, will be chosen later).

\begin{lemma}\label{3.7}
If $ M_d \ge (2d)^{d-1} $ and $ C \ge \E $, then $ n \ge 1 $.
\end{lemma}

\begin{proof}
Assume the contrary: $ M_d^{2d} \frac{ (C|\la|)^{2d} }{ v^{d-1} }
\log^{2d(d-1)} \frac{ \sqrt v }{ C|\la| } \le 1 $, that is,
\begin{equation}\label{3.8}
M_d \frac{ C|\la| }{ \sqrt{S(v)} } \log^{d-1} \frac{ \sqrt v }{ C|\la|
} \le 1 \, .
\end{equation}
For $d=1$ it means $ M_1 C |\la| \le 1 $ in contradiction to $ C|\la|>1
$ and $ M_1 \ge 1 $. Assume $ d \ge 2 $. Using \eqref{3.**},
\begin{gather*}
\frac{ \sqrt{S(v)} }{ \log^{d-1} v } < C|\la| \le
 \frac1{M_d} \sqrt{S(v)} \log^{-(d-1)} \frac{ \sqrt v }{ C|\la| } \, ; \\
\log^{d-1} v > M_d \log^{d-1} \frac{ \sqrt v }{ C|\la| } \ge \Big(
 2d \log \frac{ \sqrt v }{ C|\la| } \Big)^{d-1} \, ;
\end{gather*}
thus, $ \frac1{2d} \log v > \log \frac{\sqrt v}{C|\la|} $, that is, $
C|\la| > \sqrt{S(v)} $. Now \ref{3.8} gives $
M_d \log^{d-1} \frac{\sqrt v}{C|\la|} < 1 $, which implies $
M_d \log^{d-1} \log^d v < 1 $ (since $ C|\la| \le \frac{ \sqrt v
}{\log^d v } $ by \eqref{3.**}), and $ 2d \log \log^d v < 1 $ (since $ M_d \ge
(2d)^{d-1} $). On the other hand, $ v \ge C^d \ge \E^d $ implies
$ \log v \ge d $ and $ \log \log^d v = d \log \log v \ge d \log d $,
thus $ 2d \log \log^d v \ge 2d^2 \log d \ge 8\log2 > 1 $; a
contradiction. 
\end{proof}

From now on we assume $ M_d \ge (2d)^{d-1} $ and $ C \ge \E $ (thus, $
n \ge 1 $). We define $ v_0, \dots, v_n $ by
\[
v_k = 2^{-(n-k)} v \, .
\]

Below, ``$ y = \cO(x) $'' means that $ y \le \const \cdot x $ for some
constant dependent on $d$ only.

\begin{lemma}\label{3.9}
$ n \le \frac{ \sqrt v }{ 2 C_1 |\la| } \log^{-(d-1)} v $ provided that $C$
is large enough.
\end{lemma}

\begin{proof}
It is sufficient to prove that $ n = \cO(\log v) $; then, increasing
$C$ as needed, we get $ n \le \frac{ C }{ 2 C_1 } \log
v \le \frac{ \sqrt v }{ 2 C_1 |\la| } \log^{-(d-1)} v $ since $
C|\la| \le \frac{\sqrt v}{ \log^d v } $ by \eqref{3.**}.

We have $ \frac1{C|\la|} = \cO(1) $ (since by \eqref{3.**},
$ \frac1{C|\la|} \le \frac{\log^{d-1} v}{ \sqrt{S(v)} } $, the
latter being bounded in $ v \in (1,\infty) $).
Thus, $ \log \frac{\sqrt v}{C|\la|} = \cO(\log v) $ (since $ \log
v \ge \log C^d \ge d \ge 1 $). Also, $ C|\la| \le \frac{\sqrt
v}{ \log^d v } \le \sqrt v $. Using \eqref{3.*},
\[
2^n \le 2 M_d^{2d} \frac{ (C|\la|)^{2d} }{ v^{d-1} }
\log^{2d(d-1)} \frac{ \sqrt v }{ C|\la| } = \cO ( v \log^{2d(d-1)} v )
= \cO(v^2) \, ,
\]
which implies $ n = \cO(\log v) $.
\end{proof}

Having $ n \log^{d-1} v \le \frac{ \sqrt v }{ 2 C_1 |\la| } $ (ensured
by Lemma \ref{3.9}) we define $ \la_0,\dots,\la_n $ (either all
positive or all negative) by
\[
\frac{ \sqrt2 }{ |\la_{k+1}| } - \frac1{ |\la_k| }
= \frac{ C_1 }{\sqrt{v_k}} \log^{d-1} S(2v_k) \text{ for }
k=0,\dots,n-1 \, ; \quad \text{and} \quad \la_n = \la \, .
\]
That is,
\[
\frac1{ 2^{k/2} |\la_{n-k}| } = \frac1{|\la|} - \frac{ C_1 }{\sqrt
v} \sum_{i=0}^{k-1} \log^{d-1} S(2^{-i}v) \, ;
\]
the right-hand side is positive, since
\[
\sum_{i=0}^{k-1} \log^{d-1} S(2^{-i}v) \le k \log^{d-1} S(v) \le
n \log^{d-1} v \le \frac{ \sqrt v }{ 2 C_1 |\la| } \, ;
\]
moreover, for $ k=n $ we get
\begin{equation}\label{3.***}
2^{n/2} |\la_0| \le 2 |\la| \, .
\end{equation}

\begin{lemma}\label{3.13}
If $C$ is large enough, then $ 2v_0 \ge (2C_2)^d $.
\end{lemma}

\begin{proof}
For $C$ large enough we have $ \frac{x}{\log^{d-1}x} \ge
M_d \sqrt{2C_2} $ for all $ x \in [ (d\log C)^d, \infty ) $. We take $
x = \frac{ \sqrt v }{ C |\la| } $; using \eqref{3.**}, $ x \ge \log^d
v \ge \log^d C^d = (d\log C)^d $, thus $ \frac{x}{\log^{d-1}x} \ge
M_d \sqrt{2C_2} $. Using \eqref{3.*}, $ 2v_0 = 2^{-(n-1)} v >
M_d^{-2d} \frac{ v^d }{ (C|\la|)^{2d} } \log^{-2d(d-1)} \frac{ \sqrt v
}{ C |\la| } = \( \frac1{M_d} x \log^{-(d-1)}
x \)^{2d} \ge \( \sqrt{2C_2} \)^{2d} $.
\end{proof}

From now on we assume that $C$ is large enough, so that $ 2v_0 \ge
(2C_2)^d $. Now Corollary \ref{3.5} applies:
\[
\frac{ f_v (\la) }{ |\la| \sqrt v } \le \frac{ f_{v_0} (\la_0)
}{ |\la_0| \sqrt{v_0} } +
\sum_{k=0}^{n-1} \frac1{ R(2v_k) \log^{d-1} S(2v_k) } \, .
\]

\begin{lemma}
If $ M_d \ge 2 (2d)^{d-1} $, then $ C|\la_0| \le \frac{ \sqrt{S(v_0)}
}{ \log^{d-1} v_0 } $.
\end{lemma}

\begin{proof}
Assume the contrary. Using \eqref{3.***}, $ 2C \cdot 2^{-n/2}
|\la| \ge C|\la_0| > \frac{ \sqrt{S(2^{-n}v)} }{ \log^{d-1} (2^{-n}v)
} = 2^{-n/2} \cdot 2^{\frac{n}{2d}} v^{\frac{d-1}{2d}} \log^{-(d-1)}
(2^{-n}v) $; using \eqref{3.*}, $ M_d^{2d} \frac{ (C|\la|)^{2d} }{
v^{d-1} } \log^{2d(d-1)} \frac{ \sqrt v }{ C|\la| } \le 2^n < \frac{
(2C|\la|)^{2d} }{ v^{d-1} }\log^{2d(d-1)} (2^{-n}v) $ and $ \(
2d \log \frac{ \sqrt v }{ C|\la| } \)^{d-1} \le \frac12
M_d \log^{d-1} \frac{ \sqrt v }{ C|\la| } < \log^{d-1} (2^{-n}v) $.
For $d=1$ it means $ 1<1 $. Assume $ d \ge 2 $. We have
$ 2d \log \frac{ \sqrt v }{ C|\la| } < \log (2^{-n}v) $ and
$ \( \frac{ \sqrt v }{ C|\la| } \)^{2d} < 2^{-n} v $, that is, $ 2^n <
v \( \frac{ C|\la| }{ \sqrt v } \)^{2d} = \frac{ (C|\la|)^{2d} }{
v^{d-1} } $. Using \eqref{3.*}, $ M_d^{2d} \log^{2d(d-1)} \frac{ \sqrt
v }{ C|\la| } < 1 $, that is, $ M_d \log^{d-1} \frac{ \sqrt v }{
C|\la| } < 1 $, which cannot be true, as was shown in the proof of
Lemma \ref{3.7}.
\end{proof}

From now on we assume that $ M_d \ge 2 (2d)^{d-1} $, so that
$ C|\la_0| \le \frac{ \sqrt{S(v_0)} }{ \log^{d-1} v_0 }
$. Also, $ v_0 \ge C_2^d $ by \ref{3.13}, and $ C_2 \le C $.
Now \eqref{3.1} applies: $ f_{v_0} (\la_0) \le C_2 \la_0^2 $; and
therefore,
\begin{equation}\label{3.****}
\frac{ f_v (\la) }{ |\la| \sqrt v } \le \frac{ C_2 |\la_0|
}{ \sqrt{v_0} } + \sum_{k=0}^{n-1} \frac1{ R(2v_k) \log^{d-1}
S(2v_k) } \, .
\end{equation}

\begin{lemma}\label{3.16}
$ \sum_{k=0}^{n-1} \frac1{ R(2v_k) \log^{d-1} S(2v_k) }
= \cO \( \frac{C|\la|}{\sqrt v} \) $.
\end{lemma}

\begin{proof}
We rewrite the sum as $ \sum_{k=1}^n (2^{-(n-k)} v)^{-1/d}
\log^{-(d-1)} S(v_k) \le \linebreak
\( \log^{-(d-1)} S(2v_0) \) (2^{-n}
v)^{-1/d} \sum_{k=1}^\infty 2^{-k/d} $, note that $ \log^{-(d-1)}
S(2v_0) = \cO(1) $ (since $ C_2 > 1 $ and $ 2v_0 \ge (2C_2)^d \ge 2 $)
and see that the given sum is $ \cO \( \(\frac{2^n}{v}\)^{1/d} \)
$. By \eqref{3.*}, $ \(\frac{2^n}{v}\)^{1/d} < 2^{1/d} M_d^2 \frac{
(C|\la|)^2 }{ v } \log^{2d-2} \frac{ \sqrt v}{ C|\la| }
=\linebreak
\cO \( \( \frac{ C|\la| }{ \sqrt v } \)^2 \log^{2d-2} \frac{ \sqrt
v}{ C|\la| } \) $. It remains to note that $ \frac{ C|\la| }{ \sqrt v
} \log^{2d-2} \frac{ \sqrt v}{ C|\la| } = \cO(1) $, since the function
$ x \mapsto \frac{ \log^{2d-2} x }{ x } $ is bounded on $ [1,\infty) $
and, using \eqref{3.**}, $ \frac{ C|\la| }{ \sqrt v } \le \frac1{\log^d
v} \le \frac1{(d\log C)^d} \le \frac1{d^d} \le 1 $.
\end{proof}

\begin{proof}[Proof of Prop.~\ref{3.6}]
By \eqref{3.****} and Lemma \ref{3.16}, $ \frac{ f_v (\la) }{
|\la| \sqrt v } \le \frac{ C_2 |\la_0| }{ \sqrt{v_0} } + N_d C \frac{
|\la| }{ \sqrt v } $ for some constant $ N_d $ dependent on $d$
only. By \eqref{3.***}, $ \frac{ |\la_0| }{ \sqrt{v_0} } \le \frac{
2|\la| }{ 2^{n/2} \sqrt{v_0} } = \frac{ 2|\la| }{ \sqrt v } $. Thus, $
f_v(\la) \le |\la| \sqrt v \( C_2 \frac{ 2|\la| }{ \sqrt v } +
N_d C \frac{ |\la| }{ \sqrt v } \) \le (2 C_2 + N_d C) \la^2 $.
And finally, $ f_{v,2C_2+N_d C}(\la) \le f_{v,C_2}(\la) = f_v(\la) $.
\end{proof}

Theorem \ref{th} is thus proved.

\bigskip
\filbreak
{
\small
\begin{sc}
\parindent=0pt\baselineskip=12pt
\parbox{4in}{
Boris Tsirelson\\
School of Mathematics\\
Tel Aviv University\\
Tel Aviv 69978, Israel
\smallskip
\par\quad\href{mailto:tsirel@post.tau.ac.il}{\tt
 mailto:tsirel@post.tau.ac.il}
\par\quad\href{http://www.tau.ac.il/~tsirel/}{\tt
 http://www.tau.ac.il/\textasciitilde tsirel/}
}

\end{sc}
}
\filbreak


\begin{thebibliography}{8.}

{\raggedright
\bibitem{I} B. Tsirelson (2016):
\emph{Linear response and moderate deviations:\\ hierarchical
approach. I.}
\href{http://arxiv.org/abs/1612.08396}{arXiv:1612.08396}.
\bibitem{III} B. Tsirelson (2017):
\emph{Linear response and moderate deviations:\\ hierarchical
approach. III.}
\href{http://arxiv.org/abs/1612.08396}{arXiv:1711.05247}.

}
\end{thebibliography}
\end{document}